\documentclass[11pt]{amsart}
\usepackage{amsmath,amssymb,latexsym,soul,cite,mathrsfs}

\usepackage{color,enumitem,graphicx}
\usepackage[colorlinks=true,urlcolor=blue,
citecolor=red,linkcolor=blue,linktocpage,pdfpagelabels,
bookmarksnumbered,bookmarksopen]{hyperref}
\usepackage[english]{babel}

\usepackage[left=2.9cm,right=2.9cm,top=2.8cm,bottom=2.8cm]{geometry}
\usepackage[hyperpageref]{backref}

\usepackage[colorinlistoftodos]{todonotes}
\makeatletter
\providecommand\@dotsep{5}
\def\listtodoname{List of Todos}
\def\listoftodos{\@starttoc{tdo}\listtodoname}
\makeatother

\numberwithin{equation}{section}

\def\R {{\rm I}\hskip -0.85mm{\rm R}}

\newtheorem{theorem}{Theorem}[section]
\newtheorem{proposition}[theorem]{Proposition}
\newtheorem{lemma}[theorem]{Lemma}

\newtheorem{remark}{Remark}

\title[Positive solutions for Kirchhoff problems with vanishing nonlocal term ]
{Positive solutions for a  Kirchhoff problem with vanishing nonlocal  term }

\author[J. R. Santos Jr.]{Jo\~ao R. Santos J\'unior}
\author[G. Siciliano]{Gaetano Siciliano}

\address[J. R. Santos Jr.]{\newline\indent Faculdade de Matem\'atica
\newline\indent 
Instituto de Ci\^{e}ncias Exatas e Naturais
\newline\indent 
Universidade Federal do Par\'a
\newline\indent
Avenida Augusto corr\^{e}a 01, 66075-110, Bel\'em, PA, Brazil}
\email{\href{mailto: joaojunior@ufpa.br }{joaojunior@ufpa.br}}

\address[G. Siciliano]{\newline\indent Departamento de Matem\'atica
\newline\indent 
Instituto de Matem\'atica e Estat\'istica
\newline\indent 
 Universidade de S\~ao Paulo 
\newline\indent 
Rua do Mat\~ao 1010,  05508-090, S\~ao Paulo, SP, Brazil }
\email{\href{mailto:sicilian@ime.usp.br}{sicilian@ime.usp.br}}

\thanks{J. R. Santos J\'unior was partially
supported by CNPq-Proc. 302698/2015-9 and CAPES-Proc. 88881.120045/2016-01, Brazil. Gaetano Siciliano  was partially supported by Capes, Fapesp and CNPq, Brazil. }

\begin{document}

\maketitle
\begin{abstract}
In this paper we study the  Kirchhoff problem
\begin{equation*}
\left \{ \begin{array}{ll}
-m(\| u \|^{2})\Delta u = f(u) & \mbox{in $\Omega$,}\\
u=0 & \mbox{on $\partial\Omega$,}
\end{array}\right.
\end{equation*}
in a bounded domain, allowing the function $m$ to vanish in many different points.
Under an appropriated {\sl area condition},
by using a priori estimates, truncation techniques and variational methods, we prove a multiplicity result of positive solutions which are ordered in the $H_{0}^{1}(\Omega)$-norm.
\end{abstract}
\maketitle

\section{Introduction}

Let $\Omega\subset \mathbb R^{N}, N\geq1$ be a smooth bounded domain.
We are interested in this paper to study the existence of positive solutions for the problem
\begin{equation}\label{P}\tag{P}
\left \{ \begin{array}{ll}
-m(\| u \|^{2})\Delta u = f(u) & \mbox{in $\Omega$,}\\
u=0 & \mbox{on $\partial\Omega$,}
\end{array}\right.
\end{equation}
where $\|u\|:=|\nabla u|_{2}$ is the usual norm in the Sobolev space $H^{1}_{0}(\Omega)$
and $m:[0,\infty)\to \R$ and $f:\R\to\R$ are suitable continuous functions.
Along the paper, $|\cdot|_{p}$ will denote the $L^{p}(\Omega)-$norm.

\medskip


Problem \eqref{P} is the $N$-dimensional  version, in the stationary case, of the {\sl Kirchhoff equation} introduced in \cite{kirchhoff}. Over the past years several authors have undertaken reasonable efforts to investigate stationary Kirchhoff problems like \eqref{P}, by considering different general assumptions on functions $m$ and $f$. Without any intention to provide a survey about the subject, we would like to refer the reader to the   papers \cite{ACM,Co, HZ,MR,PZ1,PZ2} and the references therein.

Beside the importance of their contributions, all the previous mentioned papers require the function $m$
to be bounded from below by a positive constant. In this way
the problem \eqref{P} is not degenerate and  many approaches involving variational and topological methods can be used in a straightforward and effective way in order to get solutions.

On the other hand, in the recent paper \cite{AA}   the positivity assumption on $m$ is relaxed.
Indeed Ambrosetti and Arcoya consider the degenerate coefficient $m$
by allowing $m(0)=0$ and/or $\lim_{t\to+\infty}m(t)=0$.
However in such a paper the condition $m(t)>0$ for $t>0$ is mantained.
See \cite[Section 3]{AA}.

%
%

\medskip

Motivated by the above facts and by papers \cite{BB,Hess}, where multiplicity results are obtained 
for an elliptic problem under a local nonlinearity which  vanishes in different points,
a natural question concerns the existence of many solutions for problem \eqref{P}
in the case of a degenerate $m$, that is, when it  can vanish in many different points.
Indeed this is the object of this paper to which we give a positive answer. Roughly 
speaking, if $m$ vanishes in $K$ distinct points and a suitable area condition is imposed, then
the problem has $K$ positive and ordered solutions.

To be more precise, let us start to give the  assumptions on the problem.
We require the following conditions on the functions $m$ and $f$: \smallskip
\begin{enumerate}
\item[(m)]\label{m} there exist positive numbers $0<t_{1}<t_{2}<\ldots<t_{K}$ such that \smallskip
\begin{itemize}
\item $m(t_{k})=0$  for all $k\in\{1, \ldots, K\}$, \smallskip
\item $m>0 $ in  $(t_{k-1}, t_{k})$, for all  $k\in\{1, \ldots, K\};$ 
we agreed that $t_{0}=0$,\smallskip 
\end{itemize}
\item[(f)]\label{f} there exists $s_{\ast}>0$ such that $f(t)>0$ in $(0, s_{\ast})$ and $f(s_{\ast})=0$.  \medskip
\end{enumerate}

We define  the following truncation of the function $f$:
\begin{equation}\label{eq:ftrunc}
f_{\ast}(t)=\left \{ \begin{array}{ll}
f(0) & \mbox{if \ $t< 0$,}\\
f(t) & \mbox{if \ $0\leq t< s_{\ast}$,}\\
0 & \mbox{if \ $s_{\ast}\leq t$}.
\end{array}\right.
\end{equation}
which is of course continuous, and let
 $F_{\ast}(t)=\int_{0}^{t}f_{\ast}(s)ds$.

Consider  the numbers
\begin{equation}\label{eq:alphak}
\alpha_{k}:=\max_{u\in H_{0}^{1}(\Omega), \|u\|\leq t_{k}^{1/2}}\int_{\Omega}F_{\ast}(u)dx, \quad k\in\{1,\ldots K\}.
\end{equation}
It will be shown in  Lemma \ref{21} that each $\alpha_{k}\in (0, 2F(s_{*})|\Omega|)$.
Our last assumption on the data involves an {\sl area condition} on 
 $m$ and $f$, more specifically \smallskip
\begin{enumerate}
\item[(A)]\label{A} $\displaystyle\alpha_{k}<\frac12\int_{t_{k-1}}^{t_{k}}m(s)ds< F(s_{\ast})|\Omega|$, for all $k\in\{1, \ldots, K\}$.
We agree that $t_{0}=0$.
 \end{enumerate}
 Now we are able to state our main result.
\begin{theorem}\label{eq:main}
Suppose that  \emph{(m)}, \emph{(f)} and \emph{(A)}
  hold. 
Then,  problem \eqref{P} possesses at least $K$ nontrivial positive solutions. Furthermore, these solutions are ordered in the $H_{0}^{1}(\Omega)$-norm, i.e.,
$$
0<\|u_{1}\|^{2}< t_{1}< \|u_{2}\|^{2}< t_{2}<\ldots< t_{K-1}<\|u_{K}\|^{2}< t_{K}.
$$
\end{theorem}
 
It is worth  to point out now some features about our assumptions.
First of all, nothing is required to $m$ for values greater than $t_{K}$,
then no condition   is imposed on the behaviour of $m$ at infinity, where $m$ can also be negative.
%
Moreover no condition in $0$ is imposed on $m$.

Hypothesis (A)  is an  area condition  in the same spirit of the papers \cite{BB, Hess}. 
In our case it takes into account both the local and the nonlocal term.
Roughly speaking, the area under the ``bumps'' of the nonlocal term $m$ is controlled by the measure of $\Omega$
and the nonlinearity.

\medskip

To the best of our  knowledge, this is the first article in the literature to present a result of multiplicity of solutions for elliptic Kirchhoff problems in the degenerate
case.

 Our approach uses a combination of variational methods and a priori estimates.
 Indeed a suitable functional can be defined in such a way that its critical points are exactly
 solutions of \eqref{P}. As we will see, the area condition will allow us to use Mountain Pass arguments.

%
%
%
%
%
%
%
%
\medskip

The paper is organized as follows.

In Section \ref{se:prelim} we introduce some definitions and some preliminaries are given.
  In Section \ref{sec:trunc} we study an appropriated truncated problem which gives us information about the existence 
  of a solution  for the problem \eqref{P}. The proof of the main result is completed in Section \ref{sec:final}.

\medskip

\section{Preliminaries}\label{se:prelim}

We denote by $I:H_{0}^{1}(\Omega)\to\R$ the energy functional associated to problem \eqref{P}, which is given by
$$
I(u)=\frac{1}{2}M(\|u\|^{2})-\int_{\Omega}F(u)dx,
$$
where $M(t)=\int_{0}^{t}m(s)ds$, $F(t)=\int_{0}^{t}f(s)ds$. Clearly, $I\in C^{1}(H_{0}^{1}(\Omega);\R)$ and
$$
I'(u)[v]=m(\|u\|^{2})\int_{\Omega}\nabla u\nabla vdx-\int_{\Omega}f(u)vdx.
$$
Hence, critical points $u$ of $I$ are weak solutions of problem \eqref{P}, i.e.,
$$
m(\|u\|^{2})\int_{\Omega}\nabla u\nabla vdx=\int_{\Omega}f(u)vdx, \ \forall v\in H_{0}^{1}(\Omega).
$$
Let us recall the well known  Mountain Pass Theorem,
which just requires compactness at the mountain pass value,
see e.g. Theorem 6.1 in \cite{Str} and the subsequent Remark.

We recall once for all that a $C^{1}$ functional $I$ is said to satisfy the Palais-Smale condition
at level $c\in  \R$ (briefly $(PS)_{c}$ condition) if any sequence $\{u_{n}\}$
such that 
\begin{equation}\label{eq:defPS}
I(u_{n})\to c \quad \text{ and } \quad I'(u_{n})\to 0
\end{equation} 
has a convergent subsequence. A sequence satisfying \eqref{eq:defPS}
is called a $(PS)_{c}$ sequence.

\begin{theorem}\label{mp}
Let $X$ be a Banach space and $I\in C^{1}(X; \R)$. If
\begin{enumerate}
\item[$(i)$] $I(0)=0$;
\item[$(ii)$] there exist positive constants $\rho$ and $\delta$ such that $I(u)\geq \delta$ whatever $\|u\|=\rho$;
\item[$(iii)$] there exists $e\in X$ with $\|e\|\geq \rho$ and $I(e)<\delta$.
\item[$(iv)$] $I$ satisfies the $(PS)_{c}$ condition with $c=\inf_{\gamma\in \Gamma}\max_{t\in[0, 1]}I(\gamma(t))$, where
$$
\Gamma=\{\gamma\in C([0, 1], X): \gamma(0)=0 \ \mbox{and} \ \gamma(1)=e\}.
$$
Then, $c\geq \delta$ and $c$ is a critical value of $I$.
\end{enumerate}
\end{theorem}

In the next Lemma we show useful properties of the numbers $\alpha_{k}$ defined in \eqref{eq:alphak}.
\begin{lemma}\label{21}
For each $k\in\{1, \ldots, K\}$, the following hold:
\begin{enumerate}
\item[$(i)$] 
$
\displaystyle0<\alpha_{k}=\max_{u\in H_{0}^{1}(\Omega), \|u\|\leq t_{k}^{1/2}}\int_{\Omega}F_{\ast}(u)dx<F(s_{\ast})|\Omega|;
$
\item[$(ii)$] $\alpha_{1}<\alpha_{2}<\ldots<\alpha_{K}$. 
\end{enumerate}
\end{lemma}

\begin{proof}
$(i)$ Since by (f) we have
$$
\int_{\Omega}F_{\ast}(u)dx\leq F(s_{\ast})|\Omega|, \ \forall  u\in H_{0}^{1}(\Omega),
$$ 
the number $\alpha_{k}$ is well defined, it is certainly strictly positive and 
$$
\alpha_{k}\leq F(s_{\ast})|\Omega|.
$$
We show now that $\alpha_{k}$ is achieved. In fact, let us consider a sequence $\{u_{n}\}\subset H_{0}^{1}(\Omega)$ such that $\|u_{n}\|\leq t_{k}^{1/2}$ and 
$$
\int_{\Omega}F_{\ast}(u_{n})dx\to \alpha_{k}.
$$
Since $\{u_{n}\}$ is bounded in $H_{0}^{1}(\Omega)$, there is $w_{k}\in H_{0}^{1}(\Omega)$ such that, passing to a subsequence, we get
$$
u_{n}\rightharpoonup w_{k}.
$$
Using Sobolev compact embedding and the Lebesgue Dominated Convergence Theorem, it follows that
\begin{equation}\label{20}
\alpha_{k}=\int_{\Omega}F_{\ast}(w_{k})dx
\end{equation}
and then $w_{k}\neq0$, otherwise we would have $\alpha_{k}=0$.
Moreover, since the norm is weakly lower semicontinuous, we obtain $\|w_{k}\|\leq t_{k}^{1/2}$. Consequently, from \eqref{20},
$$
\alpha_{k}<F(s_{\ast})|\Omega|.
$$

$(ii)$ We assume that $K\geq2$.
It is sufficient to note that if $\int_{\Omega}F_{\ast}(w_{k-1})dx=\alpha_{k-1}$ with $\|w_{k-1}\|\leq t_{k}^{1/2}$ and $\delta_{k-1}>1$ is such that $\delta_{k-1}t_{k-1}^{1/2}=t_{k}^{1/2}$, then, by $(f_{1})$,
$$
\alpha_{k-1}=\int_{\Omega}F_{\ast}(w_{k-1})dx<\int_{\Omega}F_{\ast}(\delta_{k-1}w_{k-1})dx\leq \alpha_{k},
$$
for each $k\in\{2, \ldots, K\}$.
\end{proof}

To prove the main Theorem it will be useful to consider truncated auxiliary problems;
indeed our strategy is to obtain the multiplicity by considering $K$ different truncated problems.

\section{The  truncated problem}\label{sec:trunc}
In this section the value of $k\in \{1,\ldots, K\}$ has to be considered fixed.
Define the continuous map:
\begin{equation*}
m_{k}(t)=\left \{ \begin{array}{ll}
m(t) & \mbox{if \ $t_{k-1}\leq t< t_{k}$,}\\
0 & \mbox{otherwise},
\end{array}\right.
\end{equation*}
 agreeing  as usual that $t_{0}=0$. 

Let us consider the truncated problem:
\begin{equation}\label{Pk}\tag{$P_{k}$}
\left \{ \begin{array}{ll}
-m_{k}(\| u\|^{2})\Delta u = f_{\ast}(u) & \mbox{in $\Omega$,}\\
u=0 & \mbox{on $\partial\Omega$.}
\end{array}\right.
\end{equation}
where $f_{*}$ is defined in \eqref{eq:ftrunc}.
Of course by a weak solution of \eqref{Pk} we mean a function $u_k\in H_{0}^{1}(\Omega)$ such that
\begin{equation}\label{ws}
m_{k}(\|u_{k}\|^{2})\int_{\Omega}\nabla u_{k}\nabla v dx=\int_{\Omega}f_{\ast}(u_{k}) v dx, \ \forall  v\in H_{0}^{1}(\Omega).
\end{equation}

The energy functional associated to the problem \eqref{Pk} is $I_{k}:H_{0}^{1}(\Omega)\to\R$ defined by
$$
I_{k}(u)=\frac{1}{2}M_{k}(\|u\|^{2})-\int_{\Omega}F_{\ast}(u)dx,
$$
where $M_{k}(t)=\int_{0}^{t}m_{k}(s)ds$. It is clear that $I_{k}\in C^{1}(H_{0}^{1}(\Omega); \R)$ and critical points of $I_{k}$ are weak solutions of \eqref{Pk}. 

It is important to note that
\begin{equation}\label{22}
I_{k}(u)=\left \{ \begin{array}{ll}
I(u), & \mbox{if \ $t_{k-1}\leq \|u\|^{2}\leq t_{k} \ \mbox{and $|u|_{\infty}\leq s_{\ast}$}$,}\medskip\\
\displaystyle\frac{1}{2}\int_{t_{k-1}}^{t_{k}}m(s)ds-\int_{\Omega}F_{\ast}(u)dx, & \mbox{if \ $t_{k}\leq \|u\|^{2}$}.
\end{array}\right.
\end{equation}

\medskip

\begin{proposition}\label{4}
Suppose that \emph{(m)} and \emph{(f)} hold. If $u_{k}$ is a nontrivial weak solution of \eqref{Pk}, then 
\begin{enumerate}
\item[$(i)$]$ t_{k-1}<\|u_{k}\|^{2}<t_{k}$; \smallskip
\item[$(ii)$] $0\leq u_{k}(x)\leq s_{\ast}$, a.e. in $\Omega$.
\end{enumerate}
\end{proposition}

\begin{proof} 
$(i)$ Otherwise, we would have $m_{k}(\|u_{k}\|^{2})=0$ and, consequently
$$
\int_{\Omega}f_{\ast}(u_{k})vdx=0, \ \forall v\in H_{0}^{1}(\Omega).
$$
Thus $f_{\ast}(u_{k})=0$ a.e. in $\Omega$ and, by (f), $u_{k}=0$. Since $u\neq 0$, that leads us to contradiction. 

$(ii)$ Choosing $v=u_{k}^{-}=\min\{0, u_{k}\}$ in \eqref{ws}, we get
$$
m_{k}(\|u_{k}\|^{2})\|u_{k}^{-}\|^{2}=\int_{\Omega}f_{\ast}(u_{k}) u_{k}^{-} dx=0.
$$
By part $(i)$, we know that $m_{k}(\|u_{k}\|^{2})=m(\|u_{k}\|^{2})>0$. Consequently, $u_{k}^{-}=0$. Therefore $u_{k}=u_{k}^{+}\geq 0$. To prove that $u_{k}\leq s_{\ast}$, it is sufficient to choose $v=(u_{k}-s_{\ast})^{+}$ in \eqref{ws} and arguing in an analogous way.
\end{proof}

\begin{remark}\label{5}
In particular  Proposition \ref{4} says that if $u_{k}$ is a nontrivial weak solution of problem \eqref{Pk}, 
then $u_{k}$ is a nontrivial weak solution of problem \eqref{P}. Moreover, if $k\neq j$ then $u_{k}\neq u_{j}$.
\end{remark}

In the remaining of the Section we prove the existence of a solution $u_{k}$
for the truncated problem by using the Mountain Pass Theorem.
It is clear that we need here the area condition (A) just at the  fixed $k$, but for simplicity we
will continue to require and refer to (A).

Let us start by showing the Mountain Pass geometry. Of course, $I_{k}(0)=0.$
\begin{lemma}\label{geometry}
Suppose that \emph{(m)}, \emph{(f)} and \emph{(A)} hold.
Then, \smallskip
\begin{enumerate}
\item[$(i)$] there exist positive numbers $\delta_{k}$ and $\rho_{k}$ such that $I_{k}(u)\geq\delta_{k}$ whenever $\|u\|=\rho_{k}$; \smallskip
\item[$(ii)$] there exists $e\in H_{0}^{1}(\Omega)$ with $\|e\|>\delta_{k}$ such that $I_{k}(e)\leq0$.
\end{enumerate}
\end{lemma}
In $(ii)$ the element $e$ does not depend on $k$.
\begin{proof}  
$(i)$ From (A), \eqref{22} and Lemma \ref{21}$(i)$, it follows that for each $u\in H_{0}^{1}(\Omega)$ with $\|u\|=t_{k}^{1/2}>0$, we have
$$
I_{k}(u)=\frac{1}{2}\int_{t_{k-1}}^{t_{k}}m(s)ds-\int_{\Omega}F_{\ast}(u)dx\geq \frac{1}{2}\int_{t_{k-1}}^{t_{k}}m(s)ds-\alpha_{k}>0.
$$
The result follows now by taking $\rho_{k}=t_{k}^{1/2}$ and $\delta_{k}=\frac{1}{2}\int_{t_{k-1}}^{t_{k}}m(s)ds-\alpha_{k}$.
\medskip

$(ii)$ Fixed $u\in H_{0}^{1}(\Omega)$ with $\|u\|=1$ and $u\geq0$, it follows from (f), (A)
 and the Monotone Convergence Theorem that
$$
\lim_{t\to\infty}I_{k}(tu)=\frac{1}{2}\int_{t_{k-1}}^{t_{k}}m(s)ds-|\Omega|F(s_{\ast})<0.
$$
Thence, choosing $e=tu$ with $t$ large enough, we get $I_{k}(e)\leq0$.
\end{proof}

Then we define 
$$
c_{k}:=\inf_{\gamma\in \Gamma}\max_{t\in[0, 1]}I_{k}(\gamma(t)), 
$$
where
$$
\Gamma=\{\gamma\in C([0, 1], H^{1}_{0}(\Omega)): \gamma(0)=0 \ \mbox{and} \ \gamma(1)=e\}
$$

\medskip

\noindent and $e$ is the function obtained in the item $(iii)$ of lemma \ref{geometry}.

The next two results guarantee compactness.

\begin{lemma}\label{levels}
Suppose that  \emph{(m)} and \emph{(A)} hold. Then,
$$
c_{k}<\frac{1}{2}\int_{t_{k-1}}^{t_{k}}m(s)ds. 
$$
\end{lemma}

\begin{proof}

Define $\gamma_{\ast}(t)=te$, where $e$ is the function obtained in  item $(ii)$ of Lemma \ref{geometry}. 
It is clear that $\gamma_{\ast}\in \Gamma$. Moreover, for some $t_{\ast}\in (0, 1)$,
$$
c_{k}\leq \max_{t\in [0, 1]}I_{k}(\gamma_{\ast}(t))=I_{k}(t_{\ast}e)<\frac{1}{2}\int_{t_{k-1}}^{t_{\ast}}m(s)ds \leq \frac{1}{2}\int_{t_{k-1}}^{t_{k}}m(s)ds.
$$
The strict inequality above occurs because $e\in H_{0}^{1}(\Omega)\backslash\{0\}$ and $e\geq 0$.
\end{proof}

The next result gives the local $(PS)$ condition.

\begin{lemma}\label{ps}
Suppose that \emph{(m)} and \emph{(f)} hold. Then, the functional $I_{k}$ satisfies the $(PS)_{c}$ condition for  
$$
c\in \left(0, \frac{1}{2}\int_{t_{k-1}}^{t_{k}}m(s)ds. \right).
$$
\end{lemma}

\begin{proof}
Let $\{u_{n}\}$ be a $(PS)_{c}$ sequence for $I_{k}$. Suppose by contradiction that, 
possibly passing to  a subsequence, we have 
\begin{equation}\label{9}
\|u_{n}\|\to\infty.
\end{equation} 
Since $\{u_{n}\}$ is a $(PS)_{c}$ sequence, we have
\begin{equation}\label{17}
o_{n}(1)+ c=I(u_{n})=\frac{1}{2}M_{k}(\|u_{n}\|^{2})-\int_{\Omega}F_{\ast}(u_{n})dx
\end{equation}
and 
\begin{equation}\label{12}
o_{n}(1)=I'(u_{n})[w]=m_{k}(\|u_{n}\|^{2})\int_{\Omega}\nabla u_{n}\nabla w dx-\int_{\Omega}f_{\ast}(u_{n})wdx, \ \forall  w\in H_{0}^{1}(\Omega).
\end{equation}
It follows from \eqref{9} and \eqref{12} that
\begin{equation}\label{13}
o_{n}(1)= \int_{\Omega}f_{\ast}(u_{n})wdx, \ \forall  w\in H_{0}^{1}(\Omega).
\end{equation}
Equality in \eqref{13} implies that 
$$
u_{n}(x)\to 0 \ \mbox{a.e. in $\Omega$}.
$$
Since $F_{\ast}$ is continuous, then
\begin{equation}\label{14}
F_{\ast}(u_{n})\to 0, \ \mbox{a.e. in $\Omega$}.
\end{equation}
On the other hand
\begin{equation}\label{15}
0\leq F_{\ast}(u_{n})\leq F(s_{\ast}).
\end{equation}
From \eqref{14}, \eqref{15} and the Lebesgue Dominated Convergence Theorem, we conclude that
\begin{equation}\label{16}
\int_{\Omega}F_{\ast}(u_{n})dx\to 0.
\end{equation}
By \eqref{9}, \eqref{17} and \eqref{16}, we obtain $c=(1/2)\int_{t_{k-1}}^{t_{k}}m(s)ds$, which is a contradiction. Therefore $\{u_{n}\}$ is bounded in $H_{0}^{1}(\Omega)$ and 
passing to a subsequence, we may assume that there exists $u\in H_{0}^{1}(\Omega), t_{\ast}\geq0$ such that
\begin{equation}\label{23}
u_{n}\rightharpoonup u \quad \text{ and }\quad \|u_{n}\|\to t_{\ast}^{1/2}.
\end{equation}
Notice that $t_{k-1}\leq t_{\ast}\leq t_{k}$. In fact, if $t_{\ast}>t_{k}$ we can use exactly the same argument above to prove that $c=(1/2)\int_{t_{k-1}}^{t_{k}}m(s)ds$. On the other hand, if $t_{\ast}<t_{k-1}$ then, it follows from \eqref{17}, \eqref{23} and Lebesgue Dominated Convergence Theorem that
$$
c=-\int_{\Omega}F_{\ast}(u)dx\leq 0,
$$
leading us again to a contradiction. Thus, choosing $w=u_n$ and $w=u$, respectively, in \eqref{12}, we obtain
\begin{equation}\label{24}
m(t_{\ast})\|u\|^{2}=\int_{\Omega}f(u)u dx
\end{equation}
 and
\begin{equation}\label{25}
m(t_{\ast})t_{\ast}=\int_{\Omega}f(u)u dx.
\end{equation}
Comparing \eqref{24} and \eqref{25}, we conclude that
\begin{equation}\label{26}
\|u\|^{2}=t_{\ast}.
\end{equation}
Finally, from \eqref{23} and \eqref{26}, it follows that 
$$
u_{n}\to u \ \mbox{in} \ H_{0}^{1}(\Omega)
$$
concluding the proof.
\end{proof}

As a consequence we get the following

\begin{proposition}\label{30}
Suppose that \emph{(m)}, \emph{(f)} and \emph{(A)} hold. Then, the truncated problem \eqref{Pk} has a nontrivial solution $u_{k}$ such that:
\begin{enumerate}
\item[$(i)$] $t_{k-1}< \|u_{k}\|^{2}<t_{k}$; \smallskip
\item[$(ii)$] $0\leq u_{k}\leq s_{\ast}$; \smallskip
\item[$(iii)$] $I_{k}(u_{k})=c_{k}\geq \delta_{k}>0$.
\end{enumerate}
\end{proposition}

\begin{proof}
The existence of a nontrivial  solution $u_{k}$ satisfying $(iii)$ is a consequence of 
the Mountain Pass Theorem \ref{mp} and Lemmas \ref{geometry}, \ref{levels} and \ref{ps}. The items $(i)$ and $(ii)$ are consequences of Proposition \ref{4}, see also Remark \ref{5}.
\end{proof}

\section{Proof of Theorem \ref{eq:main}}\label{sec:final}

The proof follows by Remark \ref{5} and Proposition \ref{30}.
The fact that the solutions are positive, follows by the positivity of the nonlinearity $f$.

%
%

\end{document}